\documentclass[11pt,leqno]{amsart}
\usepackage{amssymb,amsmath,amsthm,amsfonts}
\usepackage{latexsym}
\usepackage{graphicx}
\usepackage{subfigure}
\usepackage{hyperref}
\usepackage{cite}
\usepackage{color}
\usepackage{framed}
\definecolor{shadecolor}{rgb}{.90,.95,1}
\allowdisplaybreaks

\newcommand{\nc}{\newcommand}
\nc{\les}{\lesssim}
\nc{\nit}{\noindent}
\nc{\nn}{\nonumber}
\nc{\D}{\partial}
\nc{\diff}[2]{\frac{d #1}{d #2}}
\nc{\diffn}[3]{\frac{d^{#3} #1}{d {#2}^{#3}}}
\nc{\pdiff}[2]{\frac{\partial #1}{\partial #2}}
\nc{\pdiffn}[3]{\frac{\partial^{#3} #1}{\partial{#2}^{#3}}}
\nc{\abs}[1] {\lvert #1 \rvert}
\nc{\cAc}{{\cal A}_c}
\nc{\cE}{{\cal E}}
\nc{\cF}{{\mathcal F}}
\nc{\cP}{{\cal P}}
\nc{\cV}{{\cal V}}
\nc{\cQ}{{\cal Q}}
\nc{\cGin}{{\cal G}_{\rm in}}
\nc{\cGout}{{\cal G}_{\rm out}}
\nc{\cO}{{\cal O}}
\nc{\Lav}{{\cal L}_{\rm av}}
\nc{\cL}{{\cal L}}
\nc{\cB}{{\cal B}}
\nc{\cZ}{{\cal Z}}
\nc{\cR}{{\cal R}}
\nc{\cT}{{\cal T}}
\nc{\cY}{{\cal Y}}
\nc{\cX}{{\cal X}}
\nc{\cXT}{{{\cal X}(T)}}
\nc{\cBT}{{{\cal B}(T)}}
\nc{\vD}{{\vec \mathcal{D}}}
\nc{\efield}{\mathcal{E}}
\nc{\vE}{{\vec \efield}}
\nc{\vB}{{\vec \mathcal{B}}}
\nc{\vH}{{\vec \mathcal{H}}}
\nc{\ty}{{\tilde y}}
\nc{\tu}{{\tilde u}}
\nc{\tV}{{\tilde V}}
\nc{\Pc}{{\bf P_c}}
\nc{\bx}{{\bf x}}
\nc{\bX}{{\bf X}}
\nc{\bXYZ}{{\bf XYZ}}
\nc{\bY}{{\bf Y}}
\nc{\bF}{{\bf F}}
\nc{\bS}{{\bf S}}
\nc{\dV}{{\delta V}}
\nc{\e}{{\epsilon}}
\nc{\dE}{{\delta E}}
\nc{\TT}{{\Theta}}
\nc{\dPsi}{{\delta\Psi}}
\nc{\order}{{\cal O}}
\nc{\Rout}{R_{\rm out}}
\nc{\eplus}{e_+}
\nc{\eminus}{e_-}
\nc{\epm}{e_\pm}
\nc{\eps}{\varepsilon}
\nc{\vnabla}{{\vec\nabla}}
\nc{\G}{\Gamma}
\nc{\w}{\omega}
\nc{\mh}{h}
\nc{\mg}{g}
\nc{\vphi}{\varphi}
\nc{\tlambda}{\tilde\lambda}
\nc{\be}{\begin{equation}}
\nc{\ee}{\end{equation}}
\nc{\ba}{\begin{eqnarray}}
\nc{\ea}{\end{eqnarray}}

\nc{\g}{\gamma}
\nc{\ol}{\overline}
\newcommand{\n}{\nu}

\newtheorem{theorem}{Theorem}[section]
\newtheorem{lemma}[theorem]{Lemma}
\newtheorem{prop}[theorem]{Proposition}

\def\R{\mathbb R}

\nc{\T}{\mathbb T}
\nc{\Z}{\mathbb Z}
\nc{\N}{\mathbb N}
\nc{\pt}{\partial_t}
\nc{\la}{\langle}
\nc{\ra}{\rangle}
\nc{\infint}{\int_{-\infty}^{\infty}}
\nc{\halfwidth}{6.5cm}
\nc{\figwidth}{10cm}

\nc{\nlayers}{L} \nc{\nsectors}{M}
\nc{\indicator}{\mathbf{1}}
\nc{\Rhole}{R_{\rm hole}}
\nc{\Rring}{R_{\rm ring}}
\nc{\neff}{n_{\rm eff}}
\nc{\Frem}{F_{\rm rem}}
\nc{\DD}{\Delta}
\nc{\cD}{\mathcal D}
\nc{\lnorm}{\left\|}
\nc{\rnorm}{\right\|}
\nc{\rnormp}{\right\|_{\ell^{p,\eps}}}
\nc{\rar}{\rightarrow}
\nc{\sgn}{{\rm sign}}
\allowdisplaybreaks
\date{\today}

\title[]{The Structure of Global Attractors for Dissipative Zakharov Systems with Forcing on the Torus}

\author{M.~B.~Erdo\smash{\u{g}}an, J.~L.~Marzuola, K. Newhall and  N.~Tzirakis}

\address{Department of Mathematics \\
University of Illinois \\
Urbana, IL 61801}
\email{berdogan@math.uiuc.edu}

\address{Department of Mathematics, UNC-Chapel Hill \\ CB\#3250
  Phillips Hall \\ Chapel Hill, NC 27599}
\email{marzuola@math.unc.edu }  
  
  \address{Department of Mathematics, UNC-Chapel Hill \\ CB\#3250
  Phillips Hall \\ Chapel Hill, NC 27599}
\email{knewhall@email.unc.edu }  
  
\address{Department of Mathematics \\
University of Illinois \\
Urbana, IL 61801}
\email{tzirakis@math.uiuc.edu}
\begin{document}

\begin{abstract}
The Zakharov system was originally proposed to study the propagation of Langmuir waves in an ionized plasma.  In this paper, motivated by the work of the first and third authors in \cite{et2}, we numerically and analytically investigate the dynamics of the dissipative Zakharov system on the torus in $1$ dimension.  We find an interesting family of stable periodic orbits and fixed points, and explore bifurcations of those points as we take weaker and weaker dissipation.  
\end{abstract}

\maketitle

\section{Introduction}
In this paper we study the dissipative Zakharov system with forcing:
\begin{equation}\label{eq:fdzakharov}
\left\{
\begin{array}{l}
iu_{t}+  u_{xx}+i\gamma u=nu+f, \,\,\,\,  x \in \T= \R/(2\pi\Z), \,\,\,\,  t\in [0,\infty),\\
n_{tt}-n_{xx}+\delta n_t =(|u|^2)_{xx}, \\
u(x,0)=u_0(x)\in H^{1}(\mathbb T), \\
n(x,0)=n_0(x)\in L^{2}(\mathbb T), \,\,\,\,n_t(x,0)=n_1(x)\in H^{-1}(\mathbb T), \,\,\,\, f \in H^{1}(\T).
\end{array}
\right.
\end{equation}
The original Zakharov  system ($\gamma=\delta=f=0$) was proposed in \cite{vz} as  a model for the collapse of Langmuir waves in an ionized plasma.  The complex valued function $u(x,t)$ denotes the slowly varying envelope of the electric field with a prescribed frequency and the real valued function $n(x,t)$ denotes the deviation of the ion density from the equilibrium.  Smooth solutions of the Zakharov system obey the following conservation laws:
\begin{equation}
\label{mass}
\|u(t)\|_{L^2(\T)}=\|u_{0}\|_{L^2(\T)}
\end{equation}
and
\begin{equation}
\label{Energy}
E(u,n,\nu)(t)= \int_{\T}|\partial_{x}u|^2dx+\frac{1}{2}\int_{\T} n^2dx+\frac{1}{2}\int_{\T}\n^2dx+\int_{\T}n|u|^2dx=E(u_{0},n_{0},n_{1})
\end{equation}
where $\nu$ is such that $n_t=\nu_{x}$ and $\nu_t=(n+|u|^2)_x$. These conservation laws identify $H^1 \times L^2 \times H^{-1}$ as the natural energy space for the system. Local and global well-posedness in the energy space was established by Bourgain \cite{jbz}. Lower regularity optimal results were obtained by Takaoka in \cite{ht}. The well-posedness theory extends to the dissipative and forced system  without difficulty \cite{et2}.  

In \cite{et2}, the first and third authors established a smoothing property for the Zakharov system, and as a corollary they proved  the existence and smoothness of a global attractor in the energy space. For a discussion of basic facts about global attractors see \cite{temam} and \cite{et2}. 
The problem with Dirichlet boundary conditions had been considered in \cite{fla} and \cite{gm} in more regular spaces than the energy space. The regularity of the attractor in Gevrey spaces with periodic boundary conditions was considered in \cite{sch}.

Here, we primarily focus on the dynamics of solutions to \eqref{eq:fdzakharov}.  For large dissipation we prove that the global attractor is a single point consisting of a unique stable stationary solution of the system.   Then, we proceed to investigate numerically the case of smaller dissipation  in the spirit of the numerical exploration of damped-forced Korteweg-de Vries equation in \cite{CabralRosa} and for the Waveguide Array Mode-Locking Model in \cite{wwsk}.  In particular, we explore equilibrium and periodic solutions, the branching of solutions, bifurcation points, period doubling and other interesting dynamical structures that arise.  

The paper is organized as follows. In Section~\ref{sec:exist}, we obtain preliminary estimates on the solutions and study the existence and uniqueness of stationary solutions. In Section~\ref{s:trivialattractor}, we prove that in the case of large dissipation, the global attractor consists of the unique stationary solution. In the remaining sections we study the small dissipation case numerically using nonlinear continuation methods.

\subsection{Acknowledgments}  
M.B.E. and N.T. were partially supported by NSF grants DMS-1201872 and DMS-0901222, respectively.  N.T. is partially supported by the University of
Illinois Research Board Grant RB-14054.  J.L.M. was partially supported by an IBM Junior Faculty development award, NSF Applied Math Grant DMS-1312874 and NSF CAREER Grant DMS-1352353, and acknowledges the Universit\"at Bielefeld as well as the University of Chicago for graciously hosting him during part of this work.  The authors wish to thank Jon Wilkening, Roy Goodman, Gideon Simpson for helpful discussions throughout the preparation of this work.  Simpson's comments especially assisted in successful the implementation of {\it AUTO} and Wilkening's assisted in implementing a preliminary version of the Adjoint Continuation Method.

\section{Existence of Stationary Solutions and Preliminary Estimates}\label{sec:exist}

We start by obtaining a simple bound on the $L^2$ norm of the solution. By multiplying the $u$-equation with $\overline u $ and integrating on $\T$ and then taking the imaginary part, we obtain
$$
\frac{d}{dt} \|u\|_2^2+2\gamma \|u\|_2^2=2\Im\int f\overline u.
$$
This implies by Gronwall's   and Cauchy-Schwarz inequalities that for sufficiently large $t$ (depending only on the $L^2$ norm of the initial data), we have
\be\label{absorb}
\|u\|_2\leq 2\frac{\|f\|_2}{\gamma}.
\ee

We now study the stationary solutions of the system \eqref{eq:fdzakharov}. 
Recall  that $n$ is real, and throughout the paper we assume that $n$ and $n_t$ are mean zero.  Let $(v, m)$ be a  stationary solution  of \eqref{eq:fdzakharov}. Taking $u_t=n_t=n_{tt}=0$ leads to 
\begin{equation}\label{eq:fdzakharovstat}
\left\{
\begin{array}{l}
  v_{xx}+i\gamma v=mv+f, \,\,\,\,  x \in {\mathbb T}, \\
 -m_{xx}  =(|v|^2)_{xx},  
\end{array}
\right.
\end{equation}
The second line of \eqref{eq:fdzakharovstat} implies that $m=-|v|^2+ax+b$. Therefore, the periodicity and the mean zero  assumption lead to $m=-|v|^2+\frac{1}{2\pi} \|v\|_2^2$. Substituting to the first equation, it suffices to study
\begin{equation}\label{eq:stat}
  \Big[\frac{\partial^2}{\partial x^2} + i\gamma -\frac1{2\pi}\|v\|_2^2+|v|^2\Big]
 v=f, \,\,\,\,  x \in {\mathbb T}.
\end{equation}
\begin{lemma} Fix $f\in L^2$ and $\gamma>0$. Any solution $v$ of \eqref{eq:stat}   satisfies the following a priori estimates
\begin{align}
\|v\|_2 &\leq \frac1\gamma \|f\|_2, \label{l2bound}\\
\|v_x\|_2&\leq C \max(\gamma^{-3}\|f\|_2^3,\gamma^{-2}\|f\|_2^2,\gamma^{-1/2}\|f\|_2). \label{h1dotbound}
\end{align}
\end{lemma}
\begin{proof}
By multiplying \eqref{eq:stat} with $\overline{v}$ and integrating on $\T$ and then taking the imaginary part of the equation we obtain that
$$\gamma\|v\|_2^2=\Im \int_\T f\overline v dx.$$
This implies \eqref{l2bound} by Cauchy-Schwarz inequality.

On the other hand, taking the real part we obtain 
$$
\int|v_x|^2 dx = \|v\|_{L^4}^4 -\frac1{2\pi} \|v\|_2^4 -\Re \int f\overline v dx.
$$
By the Gagliardo-Nirenberg and Cauchy-Schwarz inequalities, we have
$$
\|v_x\|_2^2 \leq C \big(\|v_x\|_2 \|v\|_2^3+ \|v\|_2^4 +\|f\|_2\|v\|_2\big).
$$
This and \eqref{l2bound} imply \eqref{h1dotbound}. 
\end{proof}
We now prove the existence of an $H^1$ solution $v$ of \eqref{eq:stat} for large $\gamma$ and/or small $\|f\|_{H^1}$  which is unique in a fixed ball in $H^1$. 
Uniqueness in the whole space will follow from Theorem~\ref{thm:attr} below.   
\begin{prop} Given $f\in H^1(\T)$ if  $\gamma>0$
 is sufficiently large, or for given $\gamma>0$ if $\|f\|_{H^1}$ is sufficiently small, then we have a unique solution of \eqref{eq:stat} in the ball $B:=\{v:\|v\|_{H^1}\leq \frac{2\|f\|_{H^1}}{\gamma}\}$.  Moreover $v\in H^3(\T)$.
\end{prop}
\begin{proof}
 First note that by Kato-Rellich theorem the operator 
 $\frac{\partial^2}{\partial x^2} -\frac1{2\pi}\|v\|_2^2+|v|^2$ is self adjoint on $L^2(\T)$ for 
 $v\in L^\infty(\T)\subset H^1(\T)$. Therefore the operator 
$$R_{\gamma,v}:= \frac{\partial^2}{\partial x^2} -\frac1{2\pi}\|v\|_2^2+|v|^2+i\gamma$$ is invertible on 
$L^2(\T)$ and we have
\be\label{rinverse}
\|R_{\gamma,v}^{-1}\|_{L^2\to L^2}\leq \frac1\gamma.
\ee
Let 
$$T_{\gamma,f}(v):= R_{\gamma,v}^{-1} f.$$
It suffices to prove that $T_{\gamma,f}$ has a fixed point in $H^1$. To do that we will prove that $T_{\gamma,f} $ is a contraction on the ball $B$.

By the resolvent identity, 
$$
S^{-1}-T^{-1}=S^{-1}(T-S)T^{-1},
$$
we have 
$$
R_{\gamma,v}^{-1} f=\big(\frac{\partial^2}{\partial x^2}+i\gamma\big)^{-1} f - \big(\frac{\partial^2}{\partial x^2}+i\gamma\big)^{-1} \big(  -\frac1{2\pi}\|v\|_2^2+|v|^2\big) R_{\gamma,v}^{-1} f.
$$
Therefore, we obtain 
\begin{align}\label{h1bound}
\|R_{\gamma,v}^{-1} f\|_{H^1} & \leq \frac{\|f\|_{H^1}}{\gamma}+C\la\gamma^{-1/2}\ra \gamma^{-1/2} \big\| \big(  -\frac1{2\pi}\|v\|_2^2+|v|^2\big) R_{\gamma,v }^{-1} f\big\|_{L^2} \\
& \hspace{.15cm}  \leq  \frac{\|f\|_{H^1}}{\gamma}+C\la\gamma^{-1/2}\ra \gamma^{-1/2} \big\|   -\frac1{2\pi}\|v\|_2^2+|v|^2\big\|_{L^\infty} \big\|R_{\gamma,v}^{-1} f\big\|_{L^2}  \notag \\
& \hspace{.3cm} \leq
\frac{\|f\|_{H^1}}{\gamma} + C\frac{\la\gamma^{-1/2}\ra}{\gamma^{3/2}} \|v\|_{H^1}^2 \|f\|_2   \leq
\frac{\|f\|_{H^1}}{\gamma} \big(1+ C\frac{\la\gamma^{-1/2}\ra}{\gamma^{1/2}} \|v\|_{H^1}^2\big).  \notag
\end{align}
Note that in the second to last inequality we used \eqref{rinverse}.
Let $M=\frac2\gamma\|f\|_{H^1}$. The inequality above implies that for sufficiently large $\gamma$ or for sufficiently small  $\|f\|_{H^1}$, $T_{\gamma,f}$ maps $B= \{v\in H^1: \|v\|_{H^1}\leq M\}$ into itself. Thus it suffices to prove that $T_{\gamma,f}$ is a contraction. 
Again by the resolvent identity and \eqref{h1bound}, we have (for $u, v\in B$)
\begin{align*}
\|R_{\gamma,u}^{-1}f -R_{\gamma,v}^{-1} f \|_{H^1} & = \big\|R_{\gamma,u}^{-1} \big(|v|^2-|u|^2+\frac{\|u\|_2^2-\|v\|_2^2}{2\pi}\big)R_{\gamma,v}^{-1} f \big\|_{H^1}
\\
& \leq \frac{\|\big(|v|^2-|u|^2+\frac{\|u\|_2^2-\|v\|_2^2}{2\pi}\big)R_{\gamma,v}^{-1} f\|_{H^1}}{\gamma} \big(1+ C\frac{\la\gamma^{-1/2}\ra}{\gamma^{1/2}} \|u\|_{H^1}^2\big) \\ 
& \leq \big\| |v|^2-|u|^2+\frac{\|u\|_2^2-\|v\|_2^2}{2\pi} \big\|_{H^1} \frac{    \|  f\|_{H^1}}{\gamma^2} \\
& \hspace{.5cm} \times \big(1+ C\frac{\la\gamma^{-1/2}\ra}{\gamma^{1/2}} \|v\|_{H^1}^2\big)\big(1+ C\frac{\la\gamma^{-1/2}\ra}{\gamma^{1/2}} \|u\|_{H^1}^2\big)  \\
& \leq C  M \big(1+ C\frac{\la\gamma^{-1/2}\ra}{\gamma^{1/2}}M^2\big)^2 \frac{\|f\|_{H^1}}{\gamma^2} 
\|u-v\|_{H^1}  .  
\end{align*}
Therefore $T_{\gamma,f}$ is a contraction on $B$ for small $\|f\|_{H^1}$ or large $\gamma$.
Finally by the following calculation the fix point $v\in B$ is in $H^3(\T)$,
\begin{align}\label{h3bound}
\|v\|_{H^3}=\|T_{\gamma,f}(v)\|_{H^3} & \leq \la \gamma^{-1}\ra \|f\|_{H^1}  + \la \gamma^{-1}\ra  \big\|    -\frac1{2\pi}\|v\|_2^2+|v|^2\big\|_{H^1} \|R_{\gamma,v}^{-1} f\|_{H^1}  \\
& \hspace{.1cm} \leq \la \gamma^{-1}\ra \|f\|_{H^1}  + C \la \gamma^{-1}\ra \|f\|_{H^1}^3 \gamma^{-3}, \nn 
\end{align}
using the standard elliptic estimate $\| \big(\frac{\partial^2}{\partial x^2}+i\gamma\big)^{-1} f \|_{H^3} \leq \langle \gamma^{-1} \rangle \| f \|_{H^1}$.
\end{proof}

\section{Attractor in the Case of Large Dissipation}
\label{s:trivialattractor}

Recall that the energy space is $X=H^1\times L^2\times \dot H^{-1}$. We will prove under some conditions on $\gamma, \delta, \|f\|_{H^1}$ that  all solutions of  \eqref{eq:fdzakharov} converge to the stationary solution $(v,-|v|^2+\frac1{2\pi}\|v\|_2^2,0)$ in $X$ as $t\to\infty$. This also implies the uniqueness of the stationary solution $v$ under these conditions.
\begin{theorem} \label{thm:attr} Given $\|f\|_{H^1}$ and $\delta>0$, the following statement holds if  $\gamma$ is sufficiently large. Consider $(u(0),n(0),n_t(0))\in X$ where $n(0)$ and $n_t(0)$ are mean-zero. Then,  the solution $(u,n,n_t)$ of  \eqref{eq:fdzakharov} converges to the 
stationary solution $(v,-|v|^2+\frac1{2\pi}\|v\|_2^2,0)$ in $X$ as $t\to \infty$. 
\end{theorem}
\begin{proof}
Given solution $(u,n,n_t)$ of \eqref{eq:fdzakharov}, let $$(w,z,z_t)=(u-v,n+|v|^2-\frac1{2\pi}\|v\|_2^2,n_t).$$
Note that $z$ and $z_t$ are mean-zero. The equation for $(w,z,z_t)$ is the following
\begin{equation}\label{eq:wz}
\left\{
\begin{array}{l}
iw_{t}+  w_{xx}+i\gamma w=z(w+v)-|v|^2w+\frac1{2\pi}\|v\|_2^2 w, \,\,\,\,  x \in {\mathbb T}, \,\,\,\,  t\in [0,\infty),\\
z_{tt}-z_{xx}+\delta z_t =(|w+v|^2-|v|^2)_{xx}.
\end{array}
\right.
\end{equation}
Fix $\epsilon>0$ and let 
$$
H=\big\|\partial_x^{-1} (z_t+\e  z)\big\|_2^2+\|z\|_2^2+2\|w_x\|_2^2+2\int_\T z(|w+v|^2-|v|^2)+\|w\|_2^2. 
$$
The above quantity $H$ was introduced in  \cite{fla} to obtain bounds in the energy space. We note that $H$ is bounded by a constant multiple of the  energy norm for any fixed $\epsilon$. 

We have
\begin{align*}
\frac{d}{dt}\big\|\partial_x^{-1} & (z_t+\e  z)\big\|_2^2 =
2 \int \partial_x^{-1} (z_t+\e  z) \partial_x^{-1} (z_{tt}+\e   z_t)\\
&=2 \int \partial_x^{-1} (z_t+\e  z) \partial_x^{-1}  [(z+|w+v|^2-|v|^2)_{xx}+(\e-\delta)  z_t]\\
&=-2\int (z_t+\e  z) (z+|w+v|^2-|v|^2)-2(\delta-\e) \int  \partial_x^{-1} (z_t+\e  z) \partial_x^{-1}z_t\\
&= -\frac{d}{dt} \|z\|_2^2-2\e\|z\|_2^2 -2\int z_t( |w+v|^2-|v|^2)-2\e\int  z(|w+v|^2-|v|^2)
\\& \hspace{.5cm} -2(\delta-\e) \|\partial_x^{-1}z_t\|_2^2-2\e(\delta-\e)  \int \partial_x^{-1}z \partial_x^{-1}z_t.
\end{align*}

Using 
$$
\big\|\partial_x^{-1} (z_t+\e   z)\big\|_2^2=\big\|\partial_x^{-1}  z_t \big\|_2^2+\e^2 \|\partial_x^{-1} z\|_2^2+2\e\int \partial_x^{-1} z \partial_x^{-1} z_t,$$
we obtain the following energy-type identity
\begin{align}
\label{eqn:dt1}
\frac{d}{dt}\big\|\partial_x^{-1} (z_t+\e  z)\big\|_2^2 +\frac{d}{dt} \|z\|_2^2&= -2\e\|z\|_2^2 -2\int z_t( |w+v|^2-|v|^2) \\
& \hspace{.3cm} -2\e\int  z(|w+v|^2-|v|^2)
 -(\delta-\e) \|\partial_x^{-1}z_t\|_2^2 \notag \\
 & \hspace{.6cm} -(\delta-\e)\big\|\partial_x^{-1} (z_t+\e  z)\big\|_2^2 +(\delta-\e)\e^2\|\partial_x^{-1} z\|_2^2. \notag
\end{align}
Now consider the derivative of the remaining terms in the definition of $H$:
\begin{align}
\label{eqn:dt2}
 2\frac{d}{dt}\|w_x\|_2^2 =-4\Re \int \overline{w_{xx}} w_t =-4\gamma \|w_x\|_2^2-4\Im \int \overline{w_{xx}}\big[z(w+v)-|v|^2w\big],
\end{align}
and
\begin{align}
\label{eqn:dt3}
 2\frac{d}{dt}\int z(|w+v|^2-|v|^2) & =2\int z_t(|w+v|^2-|v|^2)+4\Re\int z \overline{w_t}(w+v)  \\
& \hspace{0cm} =2\int z_t(|w+v|^2-|v|^2)+4\Im\int z \overline{w_{xx}} (w+v)  \notag \\
& \hspace{.75cm}  -4\gamma \Re\int z\overline w (w+v)
 +4\Im \int zv \overline w \big[|v|^2-\frac1{2\pi}\|v\|_2^2 \big] . \notag
\end{align}
We also have
\begin{equation}
\label{eqn:dt4}
\partial_t \|w\|_2^2=-2\gamma\|w\|_2^2+2\Im\int z\overline w v.
\end{equation}

Combining \eqref{eqn:dt1}-\eqref{eqn:dt4}, we observe
\begin{align*}
\frac{d}{dt}H & = -2\e\|z\|_2^2  -2\e\int  z(|w+v|^2-|v|^2)
 -(\delta-\e) \|\partial_x^{-1}z_t\|_2^2-(\delta-\e)\big\|\partial_x^{-1} (z_t+\e  z)\big\|_2^2 \\
 & \hspace{.1cm}  +(\delta-\e)\e^2\|\partial_x^{-1} z\|_2^2  -4\gamma \|w_x\|_2^2+4\Im \int w \overline{w_{xx}} |v|^2 
 -4\gamma \Re\int z\overline w (w+v)  \\
 & \hspace{.2cm} 
 +4\Im \int zv \overline w \big[|v|^2-\frac1{2\pi}\|v\|_2^2 \big]  
 -2\gamma\|w\|_2^2+2\Im\int z\overline w v.
\end{align*}

Hence,
\begin{align*}
\frac{d}{dt}H & =-\e H -\e\|z\|_2^2  
 -(\delta-\e) \|\partial_x^{-1}z_t\|_2^2-(\delta-2\e)\big\|\partial_x^{-1} (z_t+\e  z)\big\|_2^2 -(2\gamma-\e)\|w\|_2^2 \\
 & \hspace{.1cm} +2\Im\int z\overline w v
 +(\delta-\e)\e^2\|\partial_x^{-1} z\|_2^2  -(4\gamma -2\e)\|w_x\|_2^2 -4\Im \int w \overline{w_{x}} (|v|^2)_x \\
 & \hspace{.2cm}
 -4\gamma \Re\int z\overline w (w+v)
 +4\Im \int zv \overline w \big[|v|^2-\frac1{2\pi}\|v\|_2^2\big].
\end{align*}
Let $\e=\min(\frac1{2\delta},\frac\delta2,\gamma)$. Since $z$ is mean-zero, the choice of $\epsilon$ implies 
$$(\delta-\e)\e^2\|\partial_x^{-1} z\|_2^2 \leq \frac\epsilon2 \|z\|_2^2. $$ 
Therefore, we have
\begin{align*}
\frac{d}{dt}H & \leq -\e H -\frac\e2\|z\|_2^2   - \gamma \|w\|_2^2-2\gamma  \|w_x\|_2^2   +2 \Big|\int z\overline w v\Big|     +4 \Big|\int w \overline{w_{x}} (|v|^2)_x\Big|  \\
& \hspace{.6cm}
 +4\gamma \Big|\int z\overline w (w+v)\Big|
 +4\Big| \int zv \overline w \Big[|v|^2-\frac1{2\pi}\|v\|_2^2\Big]\Big|   \\
 & \hspace{.1cm}
 \leq -\e H -\frac\e2\|z\|_2^2   - \gamma \|w\|_2^2-2\gamma  \|w_x\|_2^2  \\
 & \hspace{.5cm} + C\Big[ \|z\|_2\|w\|_2  \|v\|_{H^1} +\|w\|_2\|w_x\|_2\|v\|_{H^2}^2+\|z\|_2\|w\|_2\|v\|_{H^1}^3  \\
 & \hspace{2.0cm} +\gamma\|z\|_2\|w\|_{H^1}\|w+v\|_2 \Big] \\
 &  = -\e H -\frac\e2\|z\|_2^2   - \gamma \|w\|_2^2-2\gamma  \|w_x\|_2^2  \\
 & \hspace{.5cm} +  \big[ \mathcal{I} + \mathcal{II}  + \mathcal{III} + \mathcal{IV}  \big].
\end{align*}
Note that by \eqref{absorb}  we have
$$
\|w+v\|_2\leq 2\frac{\|f\|_2}{\gamma}
$$
for sufficiently large $t$. 
Using this, we can bound term $\mathcal{IV}$ by 
\begin{multline*}
2C\|f\|_2 \|z\|_2 \|w\|_{H^1}\leq \frac{\e}{10} \|z\|_2^2+ \frac{C_1\|f\|_2^2}{ \e}\|w\|_{H^1}^2
\\
\leq \frac{\e}{10} \|z\|_2^2+\frac{C_1\|f\|_2^2}{\e} \|w\|_2^2+\frac{C_1\|f\|_2^2}{\e}\|w_x\|_2^2
\leq  \frac{\e}{10} \|z\|_2^2+\frac{\gamma}{10} \|w\|_2^2+\frac{\gamma}{10}\|w_x\|_2^2,
\end{multline*}
provided that $\e\gamma\gg\|f\|_2^2$. Summands $\mathcal{I}-\mathcal{III}$ can be bounded by the same right hand side provided that 
$$\e\gamma \gg \|v\|_{H^1}^2+\|v\|_{H^1}^6,\,\,\,\,\,\text{ and } \,\,\,\gamma\gg\|v\|_{H^2}^2.$$
By the estimates on $v$, we see that for fixed $\delta$ and $\|f\|_{H^1}$, if $\gamma$ is sufficiently large, we have for sufficiently large $t$,
$$
\frac{d}{dt}H \leq -\e H.
$$
This implies that $H$ goes to zero as $t\to\infty$. 

Observe that  
$$
\Big|\int_\T z(|w+v|^2-|v|^2)\Big|\leq C \|z\|_2 \|w\|_{H^1} (\|w\|_2+\|v\|_2),
$$
and that, by \eqref{absorb} and \eqref{l2bound}, (for large $\gamma$ and $t$) we have $\|w\|_2+\|v\|_2\ll 1 $.
Therefore, we have
$$
H\geq C\big( \|z_t\|_{H^{-1}}^2+\|z\|_2^2+\|w\|_{H^1}^2\big).
$$
This completes the proof. 
\end{proof}

\section{Numerical Methods for solving forward in time}
In this section we briefly describe the numerical method we use for the Schr\"odinger-Dirac model (see, e.g., \cite{et2}).  We chose this method as it is accurate to a higher order in time, and yet uses the structure of the Dirac and Schr\"odinger equations to drive the solution.
In particular, we apply the time-splitting method of \cite{KassamTrefethen} as applied to the soliton dynamics in, for instance, \cite{potter}.  Let us recall the equivalent system to \eqref{eq:fdzakharov} derived in \cite{et2}, which is
\begin{eqnarray}
\label{e:diracschrodinger}
\left\{ \begin{array}{l}
(i \partial_t +  \partial_x^2 + i \gamma ) u = \alpha_1 \Re ( n ) u + f, \\
(i \partial_t -d + i \delta) n = \alpha_2  d (|u|^2), \\
(u(x,0), n(x,0)) = (u_0, n_0)  \in H^1 \times L^2,
\end{array} \right.
\end{eqnarray}
where $d = (-\partial_{xx})^{\frac12}$ where generally $\alpha_{1,2}$ are taken to be $1$.  We include the parameters $\alpha_{1,2}$ here to mention that another possibly interesting approach to the nonlinear continuation arguments would be to begin from a linear, decoupled model since there exists an exact solution for $\alpha_{1,2} = \gamma= 0$ when $f = \sin x$ given by $(u,n) = (-\sin x, 0)$.  However, we will not pursue this family of branches here and instead will focus on the behavior in $\gamma$ and $\eta$ for fixed values of $\alpha_{1,2}$.

Note that the system  \eqref{e:diracschrodinger} can be re-written as
\begin{eqnarray*}
\partial_t \left( \begin{array}{c}
u  \\
n  \end{array} \right) = L \left( \begin{array}{c}
u \\
n  \end{array} \right) + N (u,n,f),
\end{eqnarray*}
where 
\begin{eqnarray*}
L = \left[  \begin{array}{cc}
i   \partial_x^2 - \gamma  & 0 \\
0 & -i d - \delta
\end{array} \right]
\end{eqnarray*}
and
\begin{eqnarray*}
N = \left(  \begin{array}{c}
-i  \Re ( n )  u -i f \\
-i d ( |u|^2 ) 
\end{array} \right).
\end{eqnarray*}
The algorithm takes place as a pseudospectral method on the Fourier side, though it implements integrating factor, time-splitting, fourth-order Runge-Kutta schemes and contour integration all at once.  The key idea is to look at the evolution over a time step, $h$, as the integral
\begin{eqnarray*}
\left( \begin{array}{c}
u_{m+1}  \\
n_{m+1}  \end{array} \right) = e^{L h } \left( \begin{array}{c}
u_{m}  \\
n_{m}  \end{array} \right) + e^{Lh} \int_0^h e^{-L s} N (u (t_m + s), n( t_m +s), f( t_m +s)) ds,
\end{eqnarray*}
which can be approximated using a Runge-Kutta method (see Cox-Matthews \cite{CoxMatthews}) as
\begin{eqnarray*}
\left( \begin{array}{c}
u_{m+1}  \\
n_{m+1}  \end{array} \right) & = & e^{Lh} \left( \begin{array}{c}
u_{m}  \\
n_{m}  \end{array} \right)  + h^{-2} L^{-3}  \times \\
&& \left(  \left[  -4 -Lh + e^{Lh} ( 4 - 3Lh + (Lh)^2) \right] N (u_m, n_m, f( t_m)) + \right. \\
&& + 2 \left[ 2 + Lh + e^{Lh} (-2 + Lh) \right] ( N (a_{m,1}, a_{m,2}, f( t_m + h/2)) \\
&& \hspace{1.5cm} + N(b_{m,1}, b_{m,2}, f (t_m + h/2))) \\
&& \left. + \left[  -4 - 3h - (Lh)^2 + e^{Lh} (4-Lh) \right] N(c_{m,1}, c_{m,2}, f(t_m + h)) \right),
\end{eqnarray*}
where
\begin{eqnarray*}
a_m &=& e^{L h/2} \left( \begin{array}{c}
u_{m}  \\
n_{m}  \end{array} \right) + L^{-1} ( e^{L h/2} -Id) N (u_m,n_m,f(t_m)), \\
b_m &=& e^{L h/2} \left( \begin{array}{c}
u_{m}  \\
n_{m}  \end{array} \right) + L^{-1} ( e^{L h/2} -Id) N (a_{m,1},a_{m,2},f(t_m + h/2)), \\
c_m &=& e^{L h/2} a_m + L^{-1} ( e^{L h/2} -Id) (2 N (b_{m,1},b_{m,2},f(t_m + h/2)) - N(u_m,n_m,f(t_m)).
\end{eqnarray*}
However, such an algorithm can have problems if $L$ has eigenvalues near $0$.  To avoid such problems
the algorithm is slightly modified by evaluating contour integrals over whole discs, which are approximated by appropriate Riemann sums. 

Using the forward in time solving numerical methods described in this section, we are able to locate stable equilibrium solutions for $\gamma$ sufficiently large with respect to $\eta$ in several cases.  Then, this equilibrium solution can be fed into a nonlinear continuation method such as {\it AUTO} or the Adjoint Continuation Method of \cite{wwsk} to begin solving with particular values of $\eta$, $\gamma$, $\alpha_1$ and $\alpha_2$.

 \section{Numerical Results in the Case of Small Dissipation}

To begin, in an attempt to model the non-trivial dynamics in the Zakharov system, we follow some of the ideas in \cite{CabralRosa} to   analyze a series of numerically integrated solutions of \eqref{e:diracschrodinger}.  In our numerical experiments 
we observe a great deal of energy exchange between the Schr\"odinger and Dirac solutions, hence we will focus on relatively small energy initial data in order to justify that our numerics are valid on long time scales.  If the Fourier modes become too large at the edges of the spectrum, we do not consider the solution to be appropriately accurate, hence all simulations included here will have small contributions at high frequency.  The time scale on which we integrate is generally $T = 50.0$ with the time step $h \sim 1e-4$.  For the forward solver, we will begin by taking $32$ Fourier modes on which to evolve.  In addition, our contour integrals in the numerical evaluation of $L^{-1}$ will be taken as a mean of 64 equidistributed points along the disc.

%We look for solutions with
%\begin{eqnarray*}
%u(0,x) = A \sin (x), \,\,\,\,\,
%n(0,x) = 0,\,\,\,\,\,
%\delta = \gamma, \,\,\,\,\,f = \eta \sin (x),
%\end{eqnarray*}
%where we vary
%\begin{eqnarray}
%\label{eqn:ranges}
%-10 < A < 10, \ \ .2 < \gamma < 1.0, \ \ .2 < \eta \leq 1.0.
%\end{eqnarray}

For a range of $\eta$, given $\gamma$ sufficiently large, we observe that the dynamics tend to a fixed equilibrium solution as in Section \ref{s:trivialattractor}.  However, for $\gamma$ much smaller (or $\eta$ large), we observe much richer dynamics in the phase space, particularly in the form of periodic orbits, multiple equilibrium solutions, and period doubling bifurcations.  Though our AUTO solutions are pseudospectral in nature, the orbits we find are still stable under forward integration over many periods and are in that sense quite numerically stable.

To observe the changes in behavior as we vary $\gamma$ from the case of the trivial attractor, we use two different nonlinear continuation methods, namely we use {\it AUTO}, \cite{auto}, and the adjoint-continuation method (ACM) of Ambrose-Wilkening from \cite{aw,wwsk}.  We use {\it AUTO}  to numerically continue the solution for a wide range of $\gamma$ keeping both $16$ and $32$ modes per component of the Zakharov system.  We thus observe branching of the equilibrium points, periodic orbits, period doubling,  invariant tori, etc,   see Figures \ref{f:auto1} and \ref{f:eta}.  We use a pseudospectral implementation of the right hand side with $32$ Fourier modes per component.  For more on implementing the spectral method in AUTO, see for instance \cite{Lee-Thorpe,Pang}.  As one may expect, the nature of our orbits can change as we vary $\gamma$ or $\eta$, which accounts for movement in the bifurcation diagram presented in Figures \ref{f:auto1} and \ref{f:eta}.

\begin{figure}
     \includegraphics[width=2.5in]{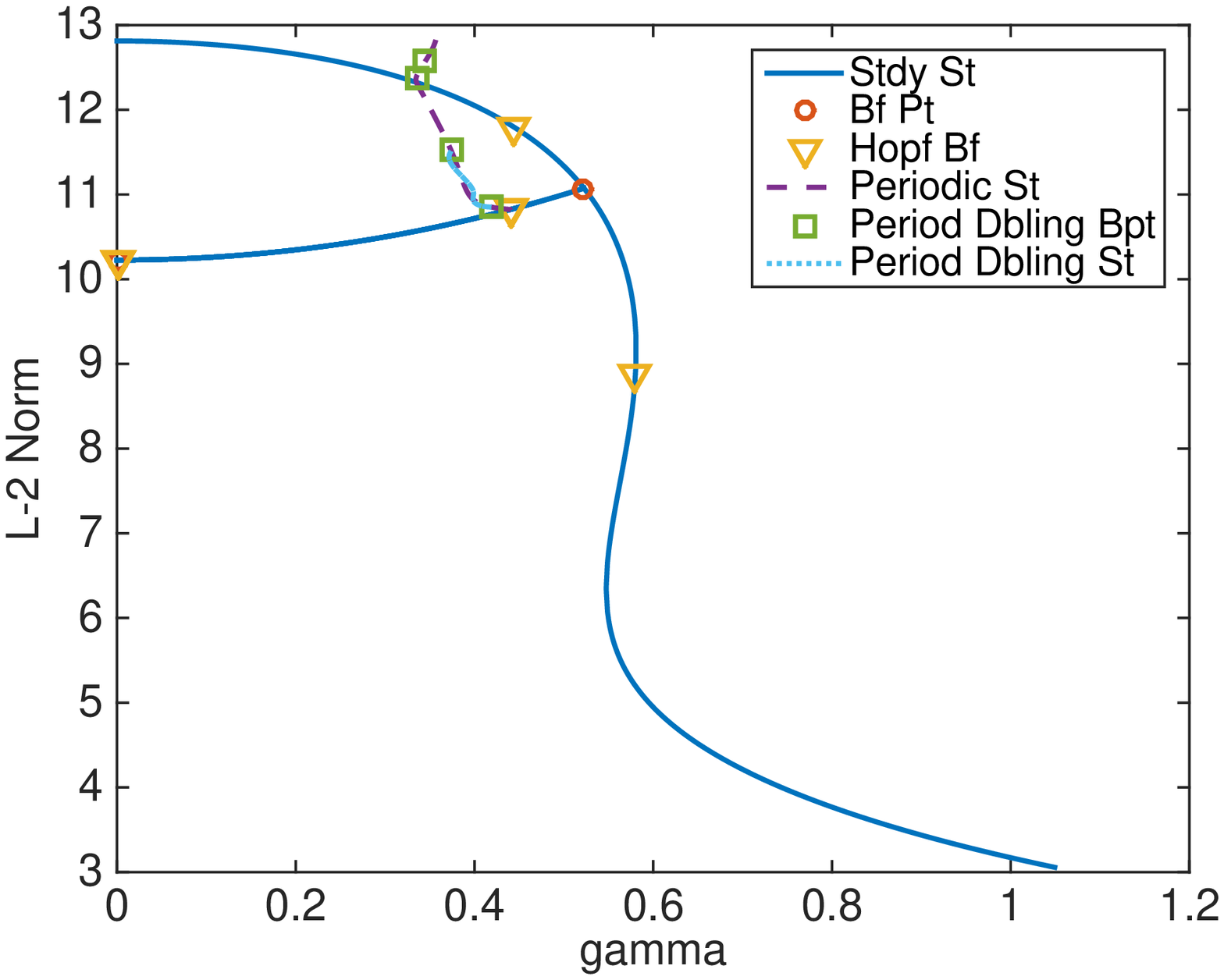}
    \includegraphics[width=2.5in]{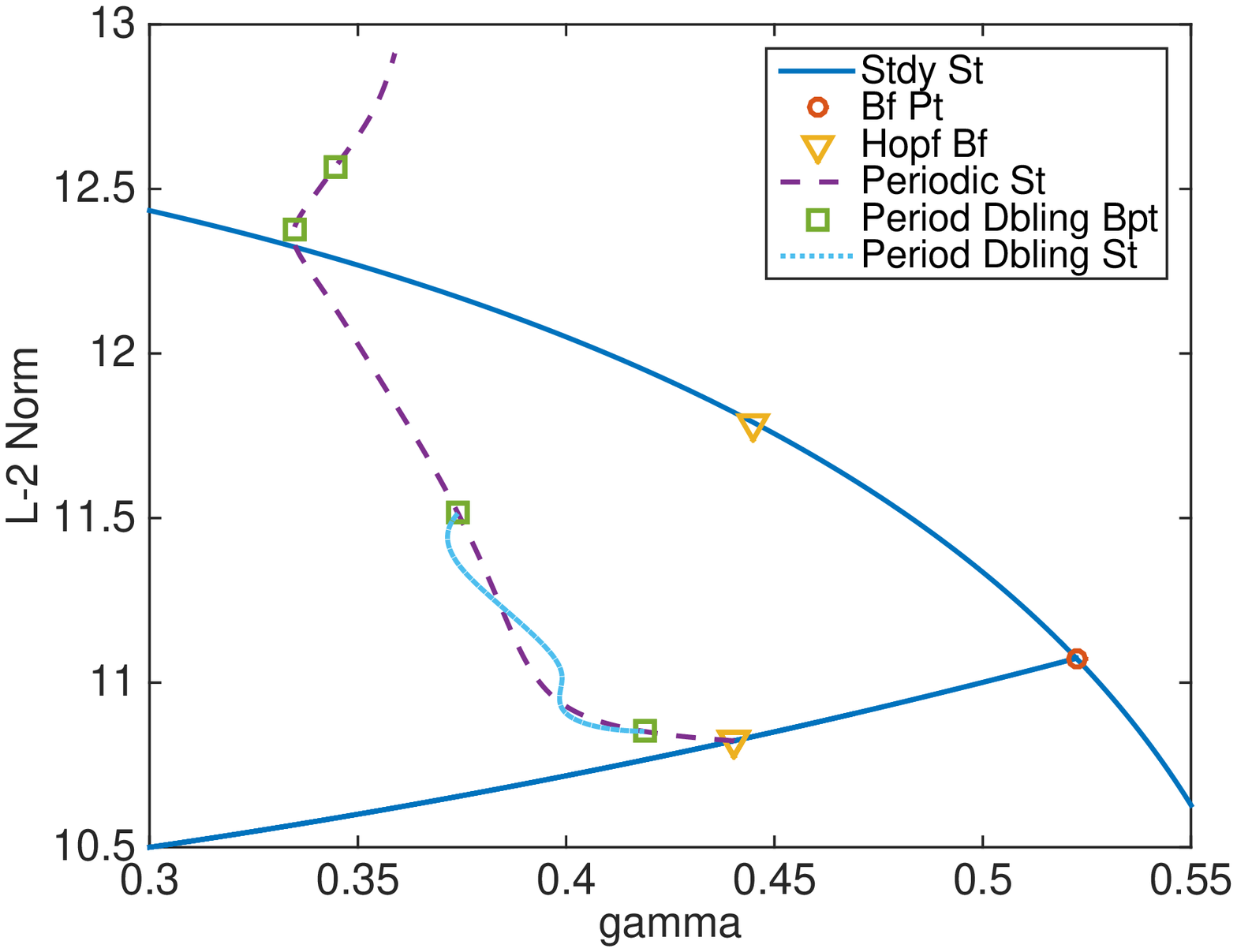}
  \caption{Using AUTO, we plot equilibrium, periodic and period doubling branches for a range of $\gamma = 0$ to $\gamma = 2.0$ with $\eta = 1.0$, $\alpha_{1,2} = 1$.  The $y$-axis is the $L^2$ norm in the case of stationary solutions and the average $L^2$ norm measured over one period for the periodic solutions.  The periodic and period doubling branches come from Hopf bifurcations and period doubling bifurcations for a pseudospectral implementation of \eqref{e:diracschrodinger} using $32$ Fourier modes on both $u$ and $n$.}
  \label{f:auto1}
\end{figure}

\begin{figure}
 \includegraphics[width=2.5in]{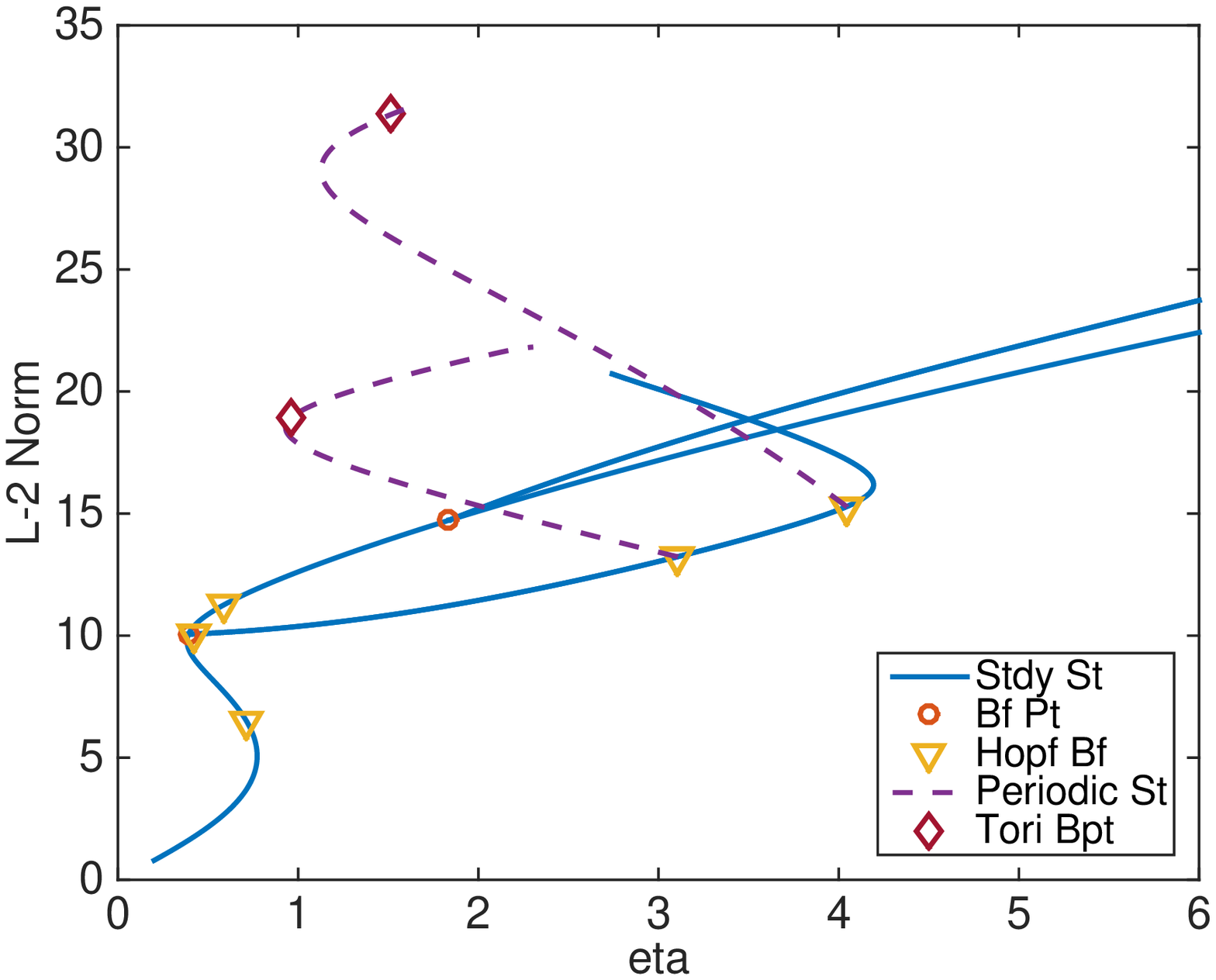}
  \includegraphics[width=2.5in]{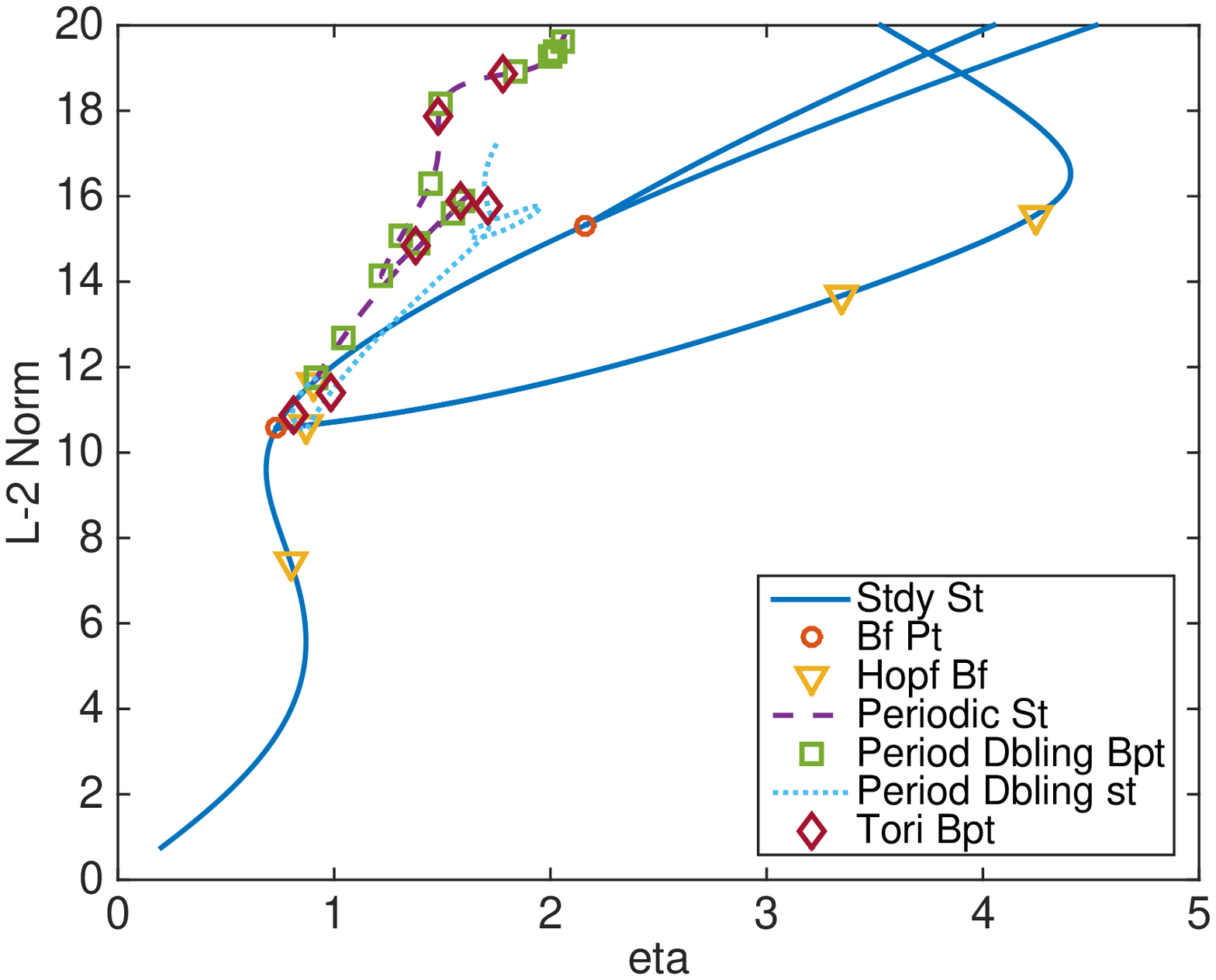}
    \includegraphics[width=2.5in]{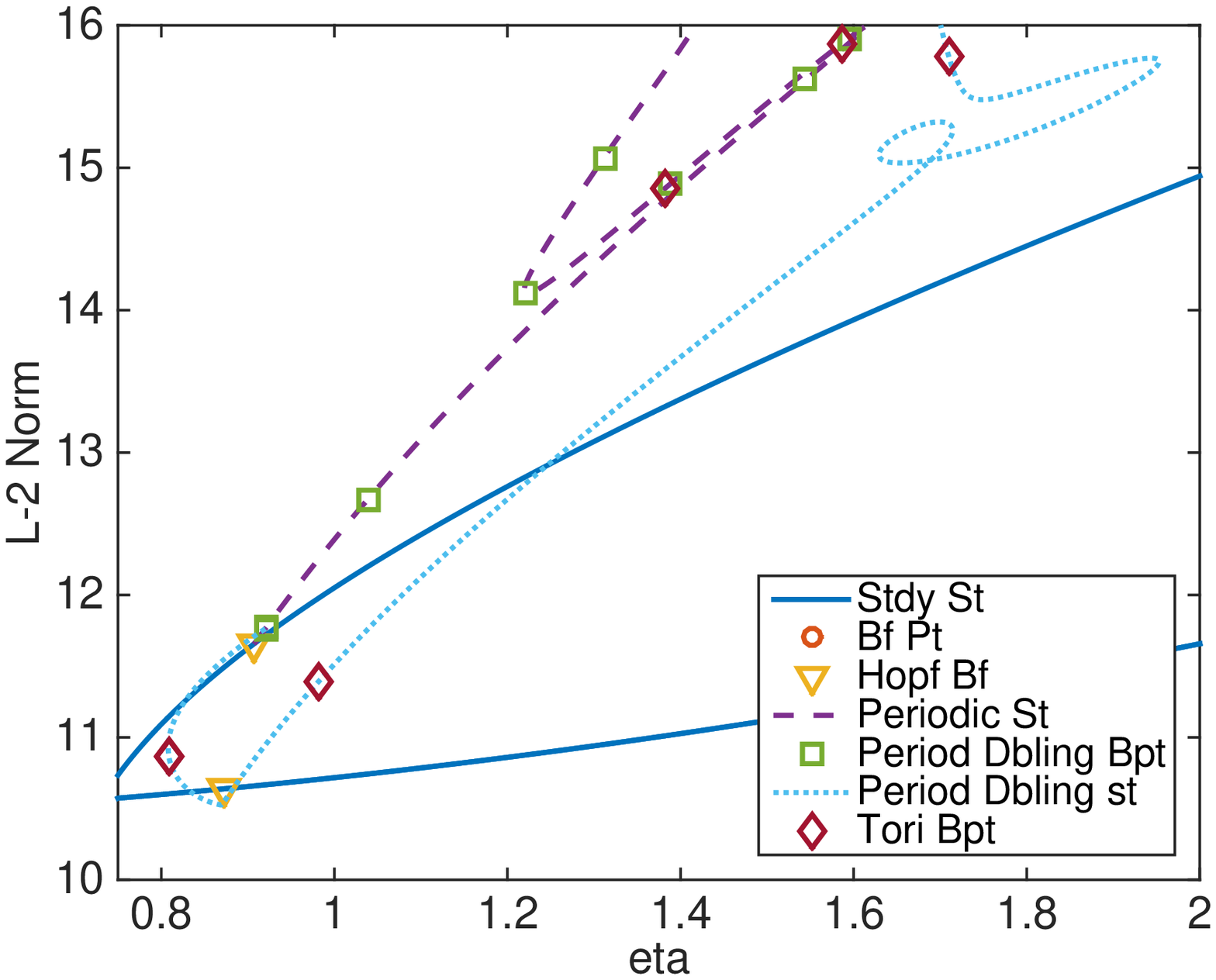}
  \caption{Top Left: Plot of the equilibrium branches and part of two periodic branches stemming from two encountered Hopf bifurcations over a range of $\eta$ from $.2$ to $14$ versus the $L^2$ norm of the total solution for $\gamma = .225$, $\alpha_1 = \alpha_2 = 1$.  Instead of period doubling bifurcations, we observe invariant tori bifurcations at this small value of $\gamma$.  Top Right:   Plot of the equilibrium branch, a periodic branch stemming from a Hopf bifurcation and a period doubling branch over a range of $\eta$ from $.2$ to $14$ versus the $L^2$ norm of the total solution for $\gamma = .4$, $\alpha_{1,2} = 1$.   Bottom: Blow up of the Top Right near the branching point.  The solutions are found using {\it AUTO}.  }
  \label{f:eta}
\end{figure}

We also implemented a simple version of the ACM method similar to that in \cite{wwsk} in {\it Matlab} in order to move along an equilibrium branch and detect a Hopf bifurcation to high accuracy.  It will be further work developing this method to efficiently study bifurcations in general, but may be worth pursuing should one wish to use many more Fourier modes and resolve more complicated types of potential orbits.    Indeed, the spatial resolution one can achieve with ACM is the primary reason to pursue other potential nonlinear continuation methods, see for instance \cite{aw}. Using a fast version of this method, we plot the spectrum of the operator linearized around the computed equilibrium solutions along a branch using the ACM method in Figure \ref{f:ACM}.  Here, we have taken $\eta = 1.0$, $\alpha_1 = 0.5$, $\alpha_2 = 1.0$ and solved over various values of $\gamma$.  

\begin{figure}
  \includegraphics[width=2.5in]{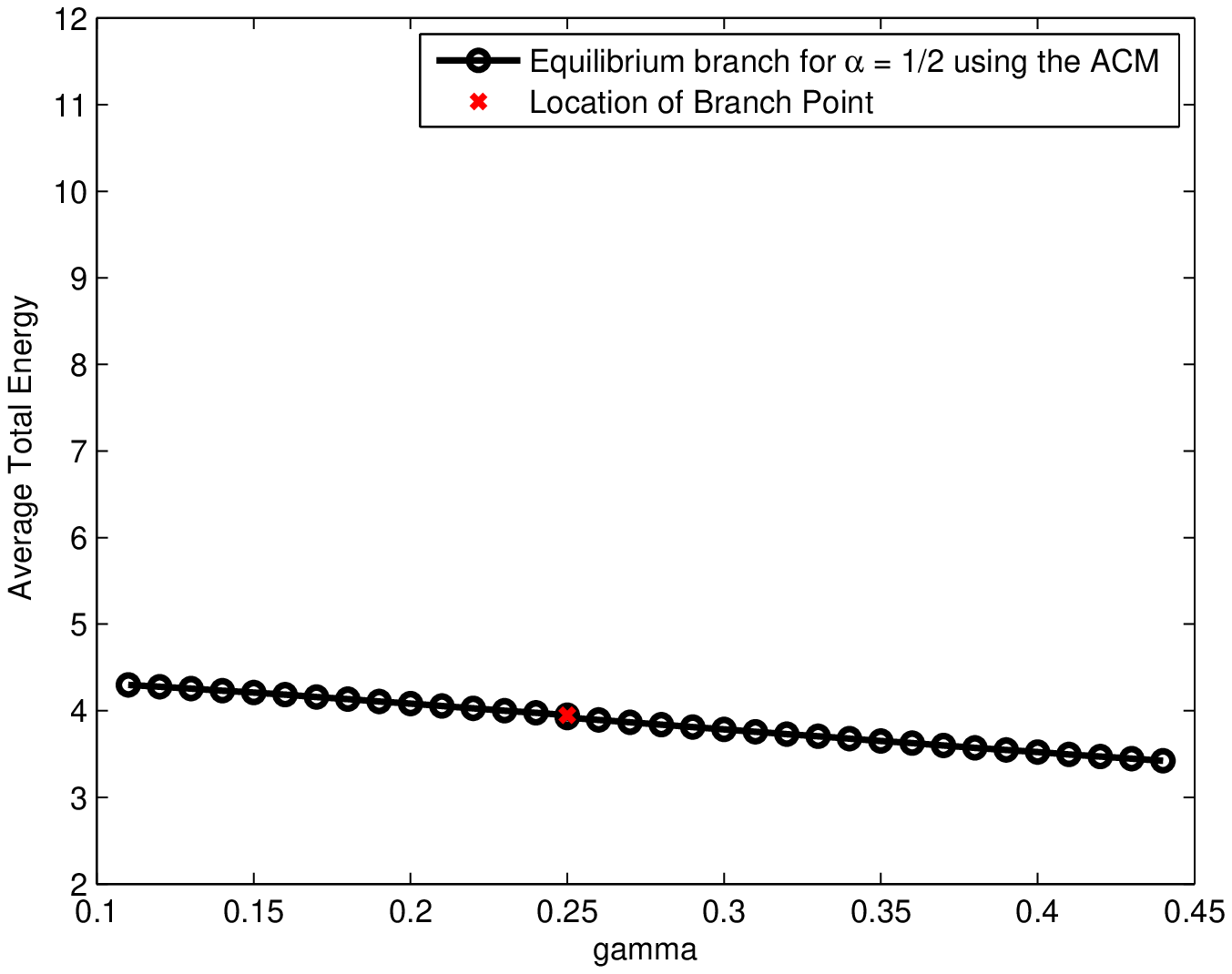}
  \includegraphics[width=3.35in]{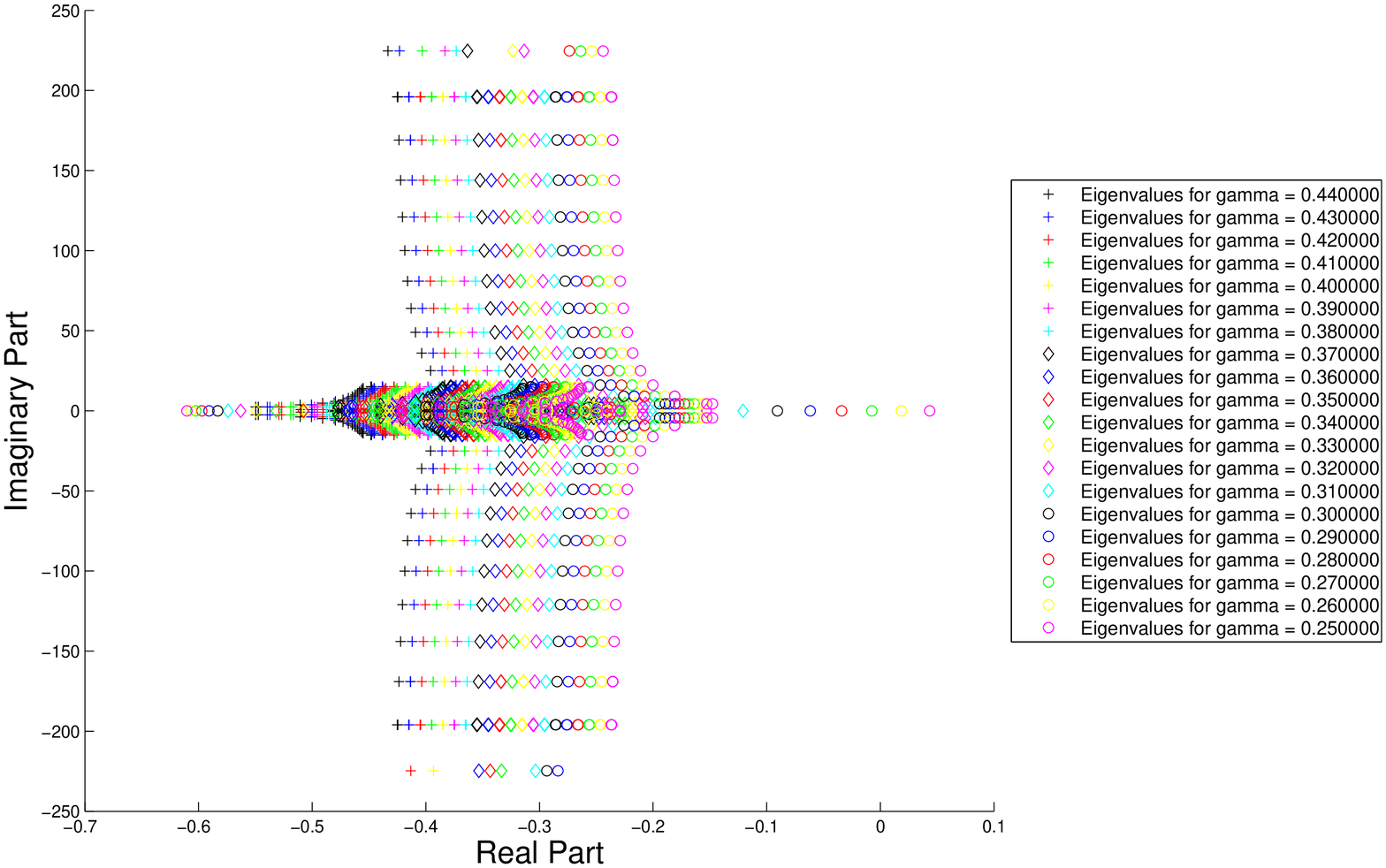}
  \includegraphics[width=3.35in]{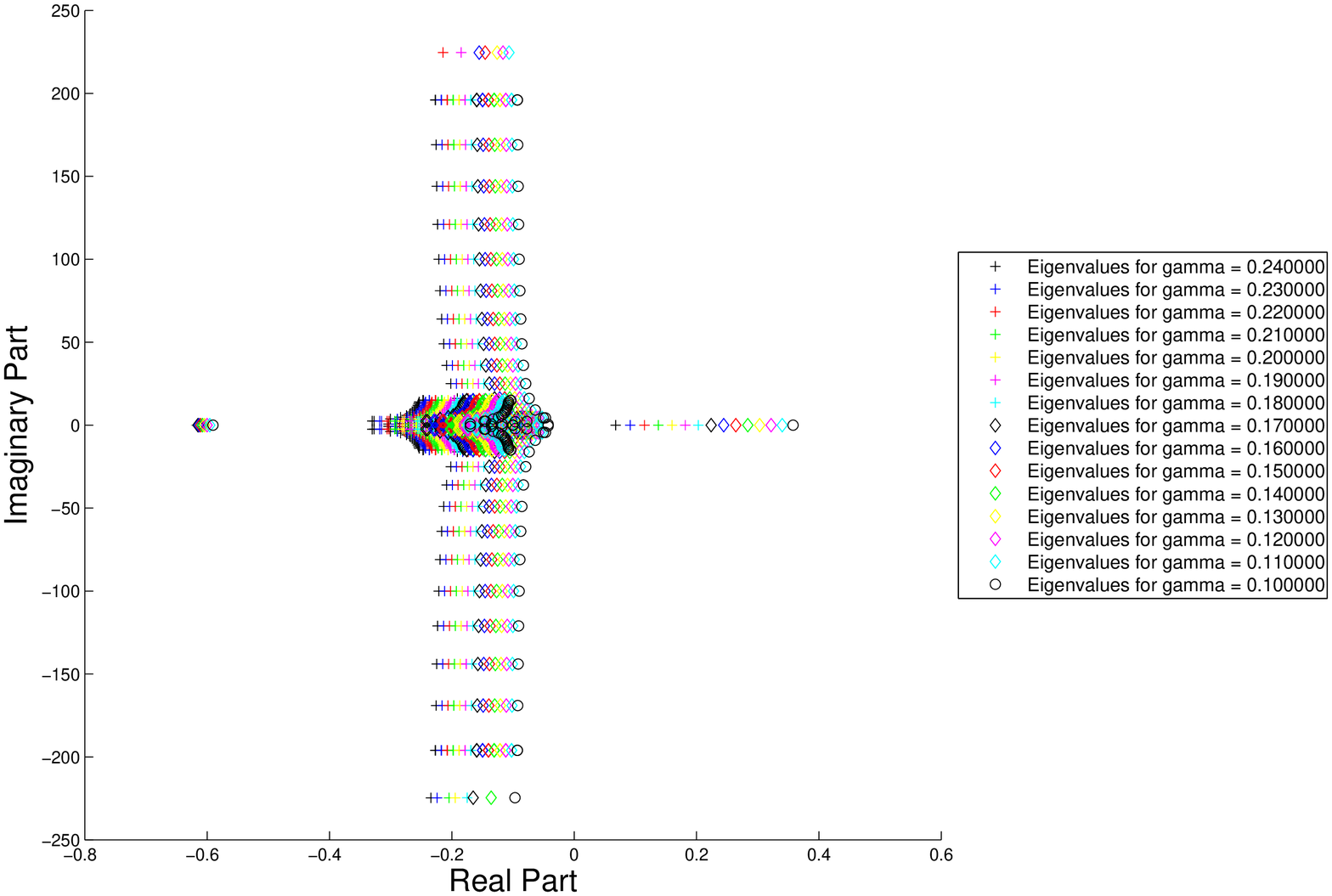}
  \caption{Plot of the lowest equilibrium branch constructed using the adjoint continuation method (top left) of \cite{aw} with $\eta = 1$, $\alpha_1 = 0.5$, $\alpha_2 = 1.0$ along with two plots of the spectrum of the linearized forward problem, for $\gamma > .25$ (top right) and $\gamma < .25$ (bottom).  The fact that the spectrum moves to the positive real part is evidence of a branch bifurcation happening near $\gamma = .25$ in the numerical continuation.}
  \label{f:ACM}
\end{figure}

The other figures present a phase plane representation of the natural energies for the Schr\"odinger and Dirac components throughout the evolution of particular solutions.  Specifically, we plot the evolution of solutions in Figure \ref{f:perorbitplots} for various values of $\eta$ corresponding to periodic branches and period doubling branches in Figure \ref{f:eta} for $\gamma = 0.4$, $\alpha_1 = \alpha_2 = 1.0$.  We also present a solution from the period doubling branch in Figure \ref{f:auto1} with $\eta = 1.0$, $\alpha_1 = \alpha_2 = 1.0$.  We plot these orbits in the energy phase plane given by $(\| u \|_{H^1} (t), \| n \|_{L^2} (t))$, which we refer to as the Schr\"odinger Energy vs. Dirac Energy phase plane coordinates motivated by \eqref{Energy}.

\begin{figure}
    \includegraphics[width=2.5in]{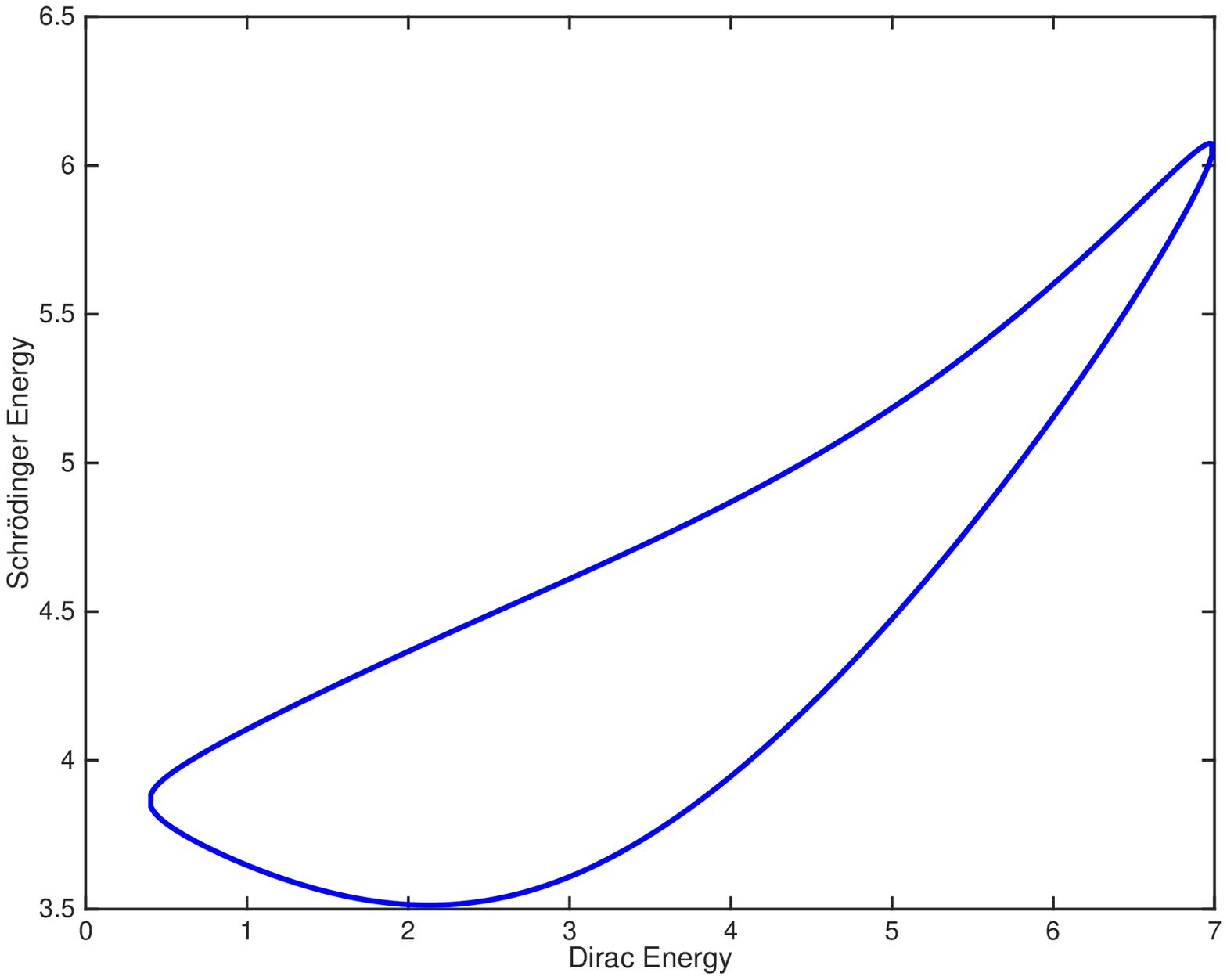}
   \includegraphics[width=2.5in]{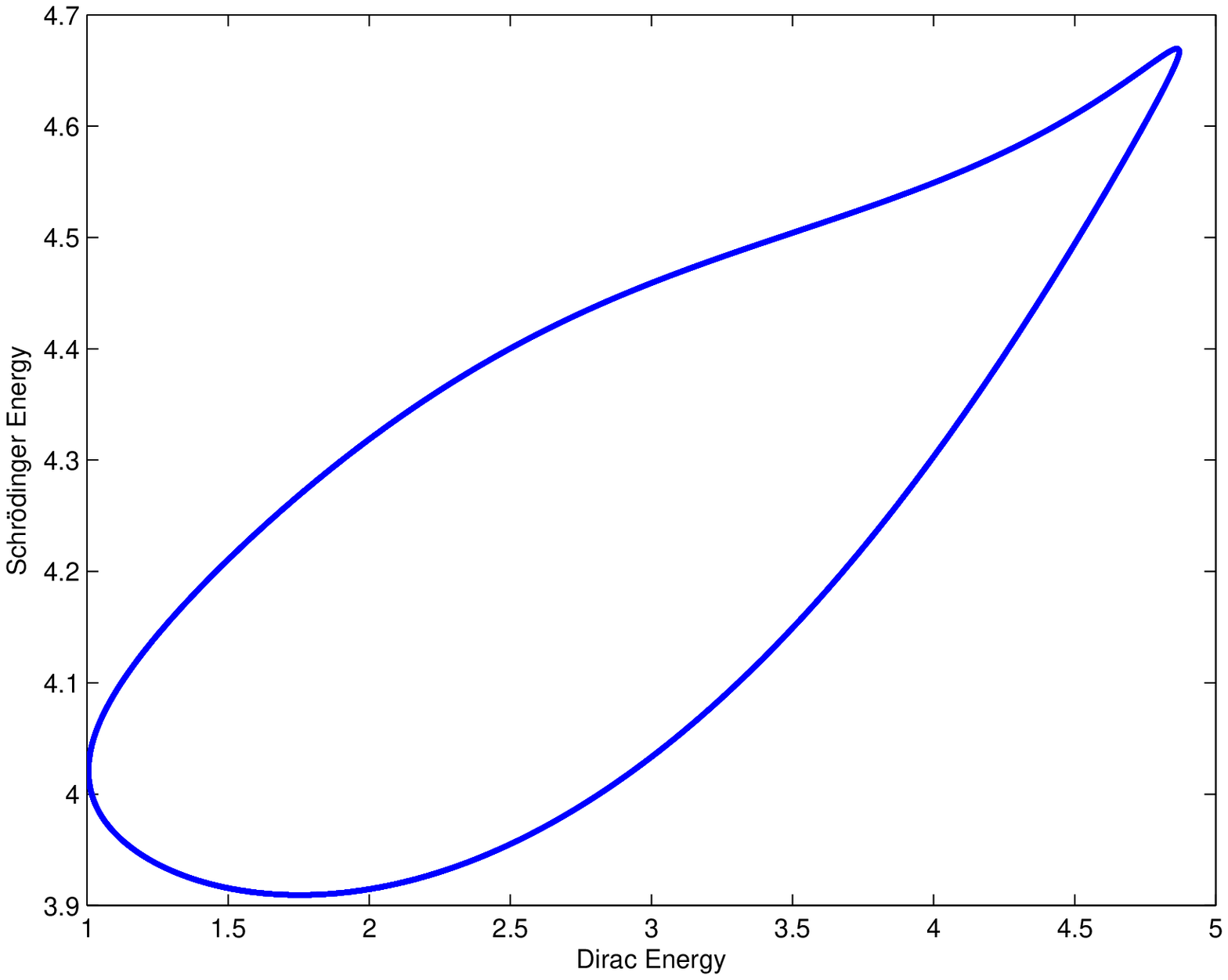}
   \includegraphics[width=2.5in]{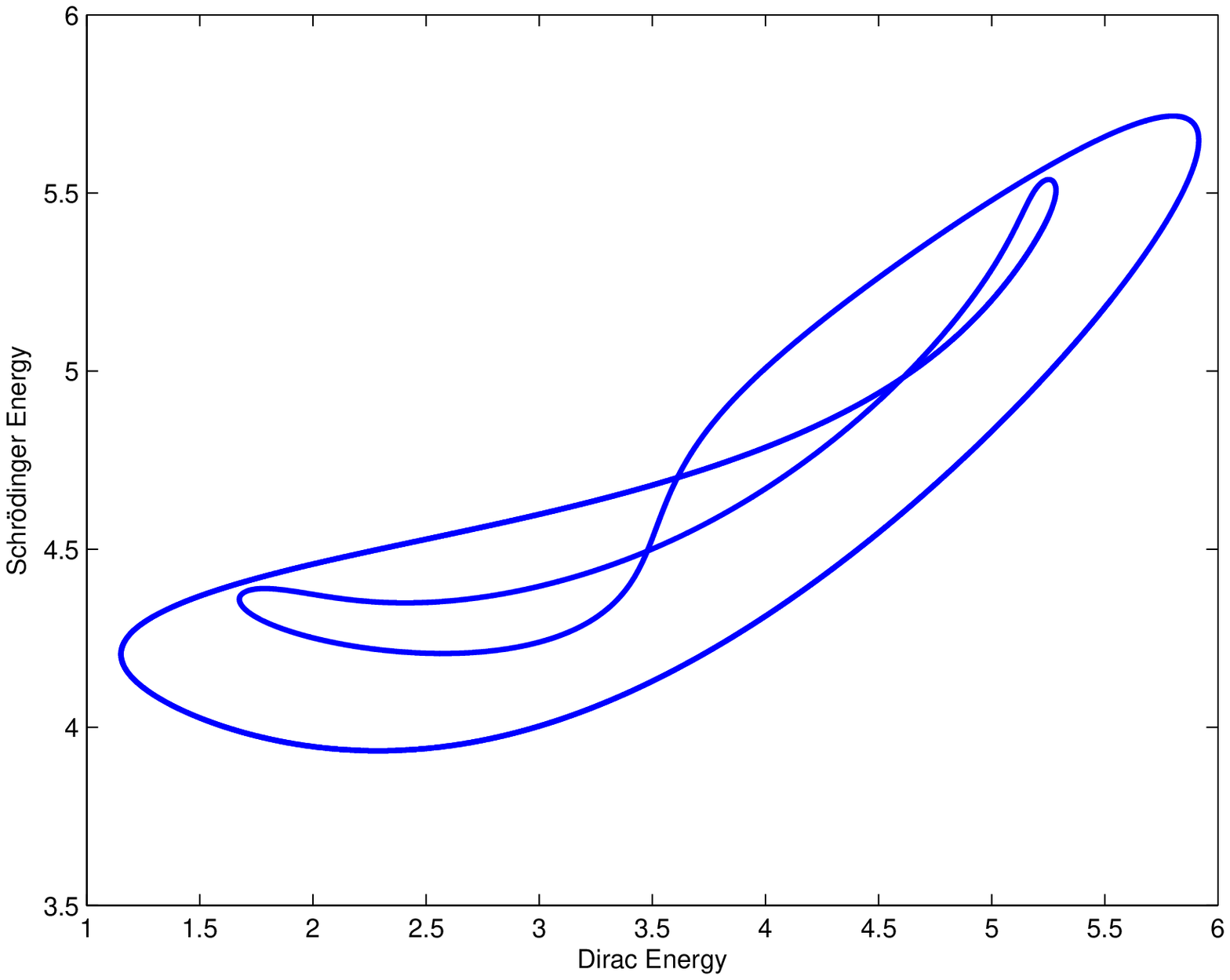}
  \caption{Top Left:  Plot  of the Dirac Final Energy versus Schr\"odinger Final Energy in the case of a periodic solution from the branch of periodic solutions in Figure \ref{f:eta} with $\gamma = .4$. Top Right:  Plot  of the Dirac Final Energy versus Schr\"odinger Final Energy in the case of a solution from the period doubling branch in Figure \ref{f:eta} with $\gamma = .4$.  Bottom:  Plot  of the Dirac Final Energy versus Schr\"odinger Final Energy in the case of a solution from the period doubling branch in Figure \ref{f:auto1} with $\gamma = .373$.}
\label{f:perorbitplots}
\end{figure}

Many people have studied numerical continuation of nonlinear states and periodic orbits in the context of NLS in the past by using shooting methods and finite difference approximations to turn infinite dimensional systems into large systems of ODEs.  See for instance \cite{SPLE,BZV} and references to them and therein.  In the recent thesis of Lee-Thorpe \cite{Lee-Thorpe}, the author implemented a spectral method in {\it AUTO}, which we have  similarly implemented here for the Zakharov system in order to take non-local operators into consideration.

\section{Discussion}

We have analytically and numerically observed rich dynamics in the dissipative periodic Zakharov system with forcing.  Open problems for future consideration include understanding the large exchange of energy from Schr\"odinger to Dirac, classifying dynamics for a larger range of energies, finding more bifurcation points, etc.  For small $\gamma$ values, we observe numerically that there is a great deal of energy transfer from the Schr\"odinger equation into the Dirac equation at the outset of the dynamics.  We have, using pseudospectral nonlinear continuation methods, discovered Hopf bifurcations, period doubling, branching of periodic orbits, and invariant tori.  However, we mention that using the orbits we have shadowed here, it is likely that further development of the adjoint continuation methods in \cite{aw} could allow one to construct nearby periodic solutions with great accuracy and hence move along the solution branches in a more robust manner.  This will be a topic of future work.


\begin{thebibliography}{100}
\bibitem{aw} D.~Ambrose and J.~Wilkening, {\em Computation of time-periodic solutions of the Benjamin-Ono equation}, J. Nonlin. Sci.,
 {\bf 20}, No. 3 (2010), 277--308.
 \bibitem{BZV} I.V. Barashenkov, E. V. Zemlyanaya, and T. C. Van Heerden, {\em Time-periodic solitons in a damped-driven nonlinear Schr\"odinger equation}, Phys. Rev. E, {\bf 83}, No. 5 (2011), 056609.
\bibitem{jbz} J.~Bourgain, {\em On the Cauchy and invariant measure problem for the periodic Zakharov system}, Duke Math J.,
 {\bf 76} (1994), 175--202.
\bibitem{CabralRosa} M.~Cabral and R.~Rosa, {\em Chaos for a damped and forced KdV equation} Physica D {\bf 192} (2004), 265--278.
\bibitem{CoxMatthews} S.M.~Cox and P.C.~Matthews, {\em Exponential time differencing for stiff systems} J. Comput. Phys. {\bf 176} (2002), 430--455.
\bibitem{auto} E. Doedel, B. Oldeman, et al, {\em AUTO-07P}.  
\bibitem{et2} M.~B.~Erdo\u{g}an and  N.~Tzirakis, {\em Long time dynamics for the forced and weakly damped Zakharov system on the torus,}  Anal. \& PDE {\bf  6-3} (2013), 723--750.
\bibitem{fla} I.~Flahaut, {\em Attractors for the dissipative Zakharov system}, Nonlinear Anal. {\bf (1991)}, 599--633. 
\bibitem{gm} O.~Goubet and I.~Moise, {\em Attractor for dissipative Zakharov system,} Nonlinear Analysis, {\bf 7} (1998), 823--847.
\bibitem{KassamTrefethen} A.-K.~Kassam and L.N. Trefethen, {\em Fourth-order time-stepping for stiff PDEs,}  SIAM J. Sci. Comput. {\bf 26}, (2005), 1214--1233.
\bibitem{Lee-Thorpe}  J. Lee-Thorpe, {\em Spectral continuation study of the temporally periodic solitons of the damped-driven nonlinear Schr\"odinger equations}, Diss. Univ. of Cape Town (2012).
\bibitem{Pang} T.~Pang,  {\em An Introduction to Computational Physics},  Cambridge University Press, New York (1997).
\bibitem{potter} T.~Potter,  {\em Effective dynamics for $N$-Solitons of the Gross-Pitaevskii equation},  J. Nonlin. Sci.  {\bf 22}, No. 3  (2001), 351--370.
\bibitem{sch} A. S. Shcherbina, {\em Gevrey regularity of the global attractor for the dissipative Zakharov system,} Dynamical Systems, Vol. {\bf 18}, {\bf 3} (2003), 201--225.
\bibitem{SPLE} K.H. Spatschek, H. Pietsch, E.W. Laedke, and T. Eickerman, {\em Chaotic behavior in time in Nonlinear-Schr\"odinger systems}, Proc. Int. Workshop on Nonlinear and Turbulent Processes in Physics, Kiev (1989).  
\bibitem{ht} H.~Takaoka, {\em Well-posedness for the Zakharov system with periodic boundary conditions,} Differential and Integral Equations, Vol. {\bf 12}, {\bf 6} (1999), 789--810.
\bibitem{temam} R.~Temam, {\em Infinite-dimensional dynamical systems in mechanics and physics}, Applied Mathematical Sciences {\bf  68}, Springer, 1997.
\bibitem{wwsk} M.O. Williams, J. Wilkening, E Shlizerman, and J.N. Kutz, {\em Continuation of periodic solutions in the waveguide array mode-locked laser}, Physica D {\bf  240} (2011), 1791-1804.
\bibitem{vz} V. E. Zakharov, {\em Collapse of Langmuir waves}, Soviet Journal of Experimental and Theoretical Physics, {\bf 35} (1972), 908--914.
\end{thebibliography}
 \end{document}